\documentclass[11pt]{amsart}
\usepackage{amssymb, latexsym, amsmath, amsfonts, tikz}
\usepackage{hyperref, enumitem, fancyhdr}
\usepackage{color}

\newtheorem{thm}{Theorem}[section]
\newtheorem{cor}[thm]{Corollary}

\newtheorem{prop}[thm]{Proposition}
\theoremstyle{definition}

\theoremstyle{remark}
\newtheorem{rem}[thm]{Remark}
\numberwithin{equation}{section}
\theoremstyle{remark}

\newcommand{\mbb}{\mathbb}
\newcommand{\ra}{\rightarrow}

\newcommand{\pa}{\partial}

\newcommand{\sm}{\setminus}

\newcommand{\no}{\noindent}

\newcommand{\Om}{\Omega}

\newcommand{\la}{\lambda}

\newcommand{\short}[1]{\textit{Short}  $\mathbb C^{#1}$}

\begin{document}
\title{A rigidity theorem for H\'{e}non maps}
\author{Sayani Bera, Ratna Pal, Kaushal Verma}

\address{SB: School of Mathematics, Ramakrishna Mission Vivekananda Educational and Research Institute, PO Belur Math, Dist. Howrah, 
West Bengal 711202, India}
\email{sayanibera2016@gmail.com}

\address{RP: Department of Mathematics, Indian Institute of Science Education and Research, Pune, Maharashtra-411008, India}
\email{ratna.math@gmail.com}

\address{KV: Department of Mathematics, Indian Institute of Science, Bangalore 560 012, India}
\email{kverma@iisc.ac.in}

\begin{abstract}
The purpose of this note is two fold. First, we study the relation between a pair of H\'{e}non maps that share the same forward and backward non-escaping sets. Second, it is shown that there exists a continuum of \short{2}'s that are biholomorphically inequivalent and finally, we provide examples of \short{2}'s that are neither Reinhardt nor biholomorphic to Reinhardt domains.
\end{abstract}

\maketitle


\section{Introduction}

\no Let $P, Q$ be a pair of holomorphic polynomials in the plane and $J_P, J_Q$ the corresponding Julia sets. If $P$ and $Q$ commute, it is known that 
$J_P = J_Q$. Conversely, if $J_P = J_Q$, a basic question that has been studied rather extensively is to find relations between $P$ and $Q$. For example, a result of Beardon \cite{Be1}, that builds on the work of Baker--Er\"{e}menko \cite{BE}, shows that
\[
P \circ Q = \sigma \circ Q \circ P
\]
where $\sigma(z) = az + b$, $\vert a \vert = 1$ and $\sigma(J_P) = J_P$. In other words, two polynomials with the same Julia set necessarily commute up to a rigid motion, of the kind described above, that preserves the shared Julia set. Other related work on this question for polynomials can be found for example in \cite{Be2} and \cite{SS}, while the case of rational functions is dealt with in \cite{LP} and \cite{D} and the references therein.

\medskip

The purpose of this note is to prove a suitable analogue of this result in 
$\mathbb{C}^2$. This will be a consequence of a more general rigidity theorem. To explain this, let $\mathcal{H}$ be the family of polynomial automorphisms of 
$\mathbb{C}^2$ of the form
\begin{equation}
H = H_m \circ H_{m-1} \circ \cdots \circ H_1
\end{equation}
where
\begin{equation}
H_j(x, y) = (b_jy, p_j(y) - \delta_j x) 
\end{equation}
with $p_j$ a polynomial of degree $d_j \ge 2$ and $b_j\delta_j \not= 0$. The degree of $H$ is $d = d_1d_2 \ldots d_m$. In what follows, the phrase {\it H\'{e}non map} will refer to a map as in (1.1). The only difference between these expressions and the more commonly accepted normal form of a H\'{e}non map lies in the non-zero scalars $b_j$. The reasons for this apparent deviation from standard nomenclature will be explained later -- they certainly do not include any implication of generality of any sort.

\medskip

Let us recall in brief the basic dynamical objects associated with such maps, keeping in mind that the extra constants $b_j$ do not really influence the calculations in a significant manner -- see \cite{BS} for details. For $R > 0$, let 
\begin{align*}
V^+_R &= \{ (x,y) \in \mathbb C^2: \vert x \vert < \vert y \vert, \vert y \vert > R \},\\
V^-_R &= \{ (x,y) \in \mathbb C^2: \vert y \vert < \vert x \vert, \vert x \vert > R \},\\
V_R &= \{ (x, y) \in \mathbb C^2: \vert x \vert, \vert y \vert \le R \}.
\end{align*}
For a given $H \in \mathcal{H}$, or for that matter a compact family of H\'{e}non maps, there exists $R > 0$ such that
\[
H(V^+_R) \subset V^+_R, \; H(V^+_R \cup V_R) \subset V^+_R \cup V_R
\]
and
\[
H^{-1}(V^-_R) \subset V^-_R, \; H^{-1}(V^-_R \cup V_R) \subset V^-_R \cup V_R.
\]
Let
\[
K^{\pm}_H = \{(x, y) \in \mathbb C^2 : \;\text{the sequence}\; \left(H^{\pm n}(x, y) \right) \; \text{is bounded} \}
\]
be the set of non-escaping points. Then $K^{\pm} \subset V_R \cup V^{\mp}_R$ and
\begin{equation}\label{escape}
\mathbb C^2 \sm K^{\pm}_H = \bigcup_{n=0}^{\infty} (H^{\mp n})(V^{\pm}_R)
\end{equation}
In fact, $\mathbb C^2 \sm K^{\pm}$ is the set of points for which 
$\Vert H^{\pm n}(x, y) \Vert \rightarrow \infty$ as $n \rightarrow \infty$. Each $H \in \mathcal{H}$ extends meromorphically to $\mathbb{P}^2$ with an isolated indeterminacy point at $I^+ = [1:0:0]$ and the same holds for $H^{-1}$ which has $I^{-} = [0:1:0]$ as its lone indeterminacy point. The Green's functions
\[
G^{\pm}_H(x, y) = \lim_{n \rightarrow \infty} \frac{1}{d^n} \log^+ \Vert H^{\pm n}(x, y) \Vert
\]
are plurisubharmonic on $\mbb C^2$ and non-negative everywhere, pluriharmonic on $\mbb C^2 \sm K^{\pm}_H$ and vanish precisely on $K^{\pm}_H$. By construction, 
\[
G^{\pm}_H \circ H = d G^{\pm}_H.
\]
Both $G^{\pm}_H$ have logarithmic growth near infinity, i.e., there exist $R>0$ and $C>0$ such that   
\begin{equation}\label{L1}
\log^+ \lvert y \rvert-C\leq G_H^+ (x,y) \leq \log^+ \lvert y \rvert+C
\end{equation}
for $(x,y)\in \overline{V_R^+ \cup V_R}$, and 
\begin{equation}\label{L2}
\log^+ \lvert x \rvert-C\leq G_H^+ (x,y) \leq \log^+ \lvert x \rvert+C.
\end{equation}
for $(x,y)\in \overline{V_R^- \cup V_R}$. Hence
\begin{equation}\label{L3}
G_H^\pm (x,y)\leq \max \{\log^+ \lvert x\rvert, \log^+ \lvert y \rvert \}+ C
\end{equation}
for all $(x,y)\in \mathbb{C}^2$. Let $J^{\pm}_H = \partial K^{\pm}_H$. It turns out that $G_H^\pm$ are the pluricomplex Green's functions for $K^{\pm}_H$ respectively. The supports of the positive closed $(1,1)$ currents 
\[
\mu^{\pm}_H = dd^c G^{\pm}_H
\]
are $J^{\pm}_H$ and $\mu_H = \mu^+_H \wedge \mu^-_H$ is an invariant measure for $H$. Since we will be dealing with more than one H\'{e}non map, these objects carry a subscript to distinguish them from their counterparts associated with other maps. Here is the main result:

\begin{thm} \label{rigidity}
Let $H$ be a H\'{e}non map. Suppose that $F \in {\rm Aut}(\mbb C^2)$ is such that $F(K^{\pm}_H) = K^{\pm}_H$.  Then $F$ is a polynomial automorphism.  Furthermore,
\begin{itemize}
\item[(i)] if $\deg(F) = 1$, then $F$ is of the form 
\[
(x,y)\mapsto (cx+d, ey+f)
\]
and
\item[(ii)] if $\deg(F)\geq 2$, then either $F$ or $F^{-1}$ is a H\'{e}non map and accordingly there exists a linear map $C(x, y) = (\delta_- x, \delta_+ y)$ with $\vert \delta_\pm \vert = 1$ such that 
\[
F^{\pm 1} \circ H = C \circ H \circ F^{\pm 1}.
\]  
In addition, there exist positive integers $m_0, n_0$ and an affine automorphism $\sigma$ such that $F^{\pm m_0}=\sigma \circ H^{n_0}$ depending on whether $F$ or $F^{-1}$ is a H\'{e}non map.
\item[(iii)] There exists a H\'{e}non map $R$ such that any H\'{e}non map $F$ satisfying $F(K_H^\pm)=K_H^\pm$ is of the form $F=\sigma_F R^{r_F}$ for some affine automorphism $\sigma_F$ and for some integer $r_F \geq 1$. 
\end{itemize}

\no Conversely, if $F$ and $H$ be two H\'{e}non maps satisfying 
\[
F\circ H=C\circ H \circ F
\]
where $C:(x,y)\mapsto (\delta_- x, \delta_+ y)$ with $\lvert \delta_\pm \rvert =1$ and $C(K_H^\pm)=K_H^\pm$, then $F(K_H^\pm)=K_H^\pm$.
\end{thm}

\begin{cor}
Let $F, H$ be H\'{e}non maps such that $K^{\pm}_H = K^{\pm}_F$. Then there exists a linear map $C(x, y) = (\delta_-x, \delta_+y)$ with $\vert \delta_{\pm} \vert = 1$ such that $F \circ H = C \circ H \circ F$. In particular, if each of $F, H$ are finite compositions of maps of the form
\[
(x, y) \mapsto (y, p(y) - x)
\]
with $p$ a monic polynomial (such maps are volume preserving), and $K^{\pm}_H = K^{\pm}_F$, then $F \circ H = H \circ F$. 
\end{cor}

In the volume preserving case, it follows from \cite{Bi} that if $d_F$ and $d_H$ are the degrees of $F, H$ respectively, then there are positive integers $k, l$ such that $d_F^k = d_H^l$ and there exists an affine automorphism $\sigma$ such that 
$F^k =\sigma \circ H^l $.

\medskip

The proof of Theorem 1.1 requires several steps and begins by using an idea of Buzzard--Fornaess \cite{BF}. The hypotheses imply that $G^{\pm}_H \circ F$ vanishes precisely on $K_H^{\pm}$ and is pluriharmonic outside it. Its restriction to vertical slices of the form $x=$ constant is harmonic away from a compact set and therefore admits a harmonic conjugate modulo a period that apriori depends on the slice under consideration. However, the pluriharmonicity of $G^{\pm} \circ F$ away from $K^{\pm}_H$ shows that the period is independent of the vertical slice and this shows that $G^{\pm+}_H \circ F$ and $G^{\pm}_H$ agree upto a multiplicative constant. The logarithmic growth of $G^{\pm}_H$ then shows that $F$  must be a polynomial. Next, by Jung's theorem, $F$ can be written as a composition of affine and elementary maps and four cases arise depending on the nature of the first and last terms in this composition. Using the mapping properties of H\'{e}non maps with respect to the filtration and the hypothesis that $F$ preserves $K^{\pm}_H$ allows us to eliminate two cases. The remaining possibilities correspond precisely to the conclusion that either $F$ or $F^{-1}$ is a H\'{e}non map. Again, $F(K^{\pm}_H) = K^{\pm}_H$ implies that $K^{\pm}_H \subset K^{\pm}_F$ and the rigidity theorem of Dinh--Sibony \cite{DS} now shows that $J^{\pm}_H = J^{\pm}_F$ and hence $G^{\pm}_H = G^{\pm}_F$. Working with the associated B\"{o}ttcher functions then shows that $F$(or $F^{-1}$) and $H$ must commute up to a linear map $C$.

\medskip

The connection to \short{2}'s comes via the observation, due to Fornaess \cite{Fo}, that the sub-level sets of the Green's functions of H\'{e}non maps are examples of \short{2}'s. Recall that a \short{2} $\subset \mbb C^2$ is a domain satisfying the following properties: It can be exhausted by an increasing sequence of biholomorphic images of the ball $\mbb B^2 \subset \mbb C^2$, the Kobayashi metric on it vanishes identically, but it admits a non-constant plurisubharmonic function that is bounded above. We show that \short{2}'s arising as sub-level sets of the Green's functions of H\'{e}non maps are not Reinhardt. Furthermore, if the H\'{e}non map is chosen as a small perturbation of a hyperbolic polynomial in one variable, then the sub-level sets of the associated Green's functions cannot even be biholomorphic to a Reinhardt domain. 

\medskip

Finally, a word about the choice of the normal form of the maps in $\mathcal H$: every map in $\mathcal{H}$ is, by Friedland--Milnor \cite{FM}, conjugate to a finite composition of maps of the form $(x, y) \mapsto (y, p(y) - ax)$ for some $a \not= 0$. The natural starting point then would be to consider the family of maps in which each $b_j = 1$ -- these are the so called generalized H\'{e}non maps. If we work with such maps, the proof of the main theorem shows that every automorphism that preserves the associated non-escaping sets must be of the form (1.1) and there is no apriori reason for concluding that $b_j = 1$. The point, therefore, is that all calculations are done in a fixed coordinate system dictated by the choice of the map we begin with. Hence, the choice of the family 
$\mathcal{H}$ as in (1.1) is not an artificial one but is naturally induced by the situation; it also serves to state the main theorem more succinctly.

\medskip

Finally, the one variable analogue of the main theorem is given in the Appendix. It supplements the existing known results on maps that share the same Julia set.


\section{B\"{o}ttcher functions and their relation with the Green's functions}
\no 
Let $H=H_m \circ H_{m-1}\circ \cdots \circ H_1$, where $H_j(x,y)=(b_jy,p_j(y)-\delta_j x)$ with $p_j$ a polynomial of degree $d_j$ with highest degree coefficient $c_j$ and $b_j \delta_j \not= 0$. Then $H(x,y)=(H_1(x,y),H_2(x,y))$ where the degree of $H_1$ is strictly less than that of $H_2$ when regarded as a polynomial in $y$ and in fact 
\begin{equation}\label{H1}
H_2(x,y)=c_H y^d+q(x,y)
\end{equation}
where 
\[
c_H = \prod_{j=1}^m {c_j}^{d_{j+1}\cdots d_m}
\]
with the convention that $d_{j+1}\cdots d_m=1$ when $j=m$, $d=d_j \cdots  d_1$ and $q$ is a polynomial in $x,y$ of degree strictly less than $d$. 

\medskip 

\no 
Similarly, $H^{-1}(x,y)=(H_1'(x,y),H_2'(x,y))$ where the degree of $H_2'$ is strictly less than that of $H_1'$ when regarded as a polynomial in $x$ and
\begin{equation}\label{H2}
H_1'(x,y)=c_H' x^d+q'(x,y)
\end{equation}
where 
\[
c_H'=\prod_{j=1}^m {\left({c_j}\delta_j^{-1}\right)}^{d_{j-1}\cdots d_1} b_j^{-d_j d_{j-1} \cdots d_1}
\]
with the convention that $d_{j-1}\cdots d_1=1$ when $j=1$ and $q'$ is a polynomial in $x,y$ of degree strictly less than $d$.

\medskip 
\no Keeping track of the various steps in the proof of Proposition 5.2 in \cite{HO} shows that:

\begin{prop}\label{Bot}
For a given H\'{e}non map $H$, there exist non-vanishing holomorphic functions 
$\phi_H^\pm: V_R^\pm \rightarrow \mathbb{C}$  such that 
\[
\phi_H^+\circ H(x,y)=c_H{(\phi_H(x,y))}^d
\]
in $V_R^+$ and 
\[
\phi_H^-\circ H^{-1}(x,y)=c_H'{(\phi_H^-(x,y))}^d
\]
in $V_R^-$. Further,
\[
\phi_H^+(x,y)\sim y \text{ as } \lVert(x,y)\rVert\rightarrow \infty \text{ in } V_R^+
\]
and
\[
\phi_H^-(x,y)\sim x \text{ as } \lVert(x,y)\rVert\rightarrow \infty \text{ in } V_R^-.
\]

\end{prop}
\begin{proof}
Let $H^n(x,y)=({(H^{n})}_1(x,y),{(H^{n})}_2(x,y))$. Note that $y_n={(H^{n})}_2(x,y)$ is a polynomial in $x$ and $y$ of degree $d^{n}$.
Consider the telescoping product
\begin{equation}\label{tele}
y.\frac{{y_1}^{\frac{1}{d}}}{y}.\cdots.\frac{y_{n+1}^{\frac{1}{d^{n+1}}}}{y_{n}^{\frac{1}{d^{n}}}}\cdots
\end{equation}
which will be shown to converge. From (\ref{H1}), 
 
\begin{eqnarray}\label{bott}
\frac{y_{n+1}^{\frac{1}{d^{n+1}}}}{y_{n}^{\frac{1}{d^{n}}}}&=&\frac{{\left(c_H y_n^d+q(x_n,y_n)\right)}^{\frac{1}{d^{n+1}}}}{y_n^{\frac{1}{d^n}}}\nonumber\\
&=&{\left(c_H+\frac{q(x_n,y_n)}{y_n^d}\right)}^{\frac{1}{d^{n+1}}}\nonumber\\
&=& {c_H}^{\frac{1}{d^{n+1}}}{\left(1+\frac{q(x_n,y_n)}{c_H y_n^d}\right)}^{\frac{1}{d^{n+1}}}.
\end{eqnarray}

Choose $R>0$ sufficiently large so that 
\[
\left\lvert \frac{q(x_n,y_n)}{c_H y_n^d}\right\rvert <1
\]
for all $\lvert y\rvert >R$ and for all $n\geq 1$.
Let ${\left(1+\frac{q(x_n,y_n)}{c_H y_n^d}\right)}^{\frac{1}{d^{n+1}}}$ be the principal branch of the $d^{n+1}$-th root of 
$\left(1+\frac{q(x_n,y_n)}{c_H y_n^d}\right)$.

\medskip 
\no 
Now note that the convergence of the product
\[
y.\frac{{y_1}^{\frac{1}{d}}}{y}.\cdots.\frac{y_{n+1}^{\frac{1}{d^{n+1}}}}{y_{n}^{\frac{1}{d^{n}}}}\cdots.
\]
is equivalent to the convergence of the series
\begin{equation}\label{Log}
{\rm{Log}}\ y+ {\rm{Log}} \left(\frac{{y_1}^{\frac{1}{d}}}{y}\right)+\cdots + {\rm{Log}}\left(\frac{y_{n+1}^{\frac{1}{d^{n+1}}}}{y_{n}^{\frac{1}{d^{n}}}}\right)+\cdots.
\end{equation}
There exists $M>0$ (depending on the coefficients of the polynomial $q$) such that
\begin{equation}\label{zero}
\left\lvert\frac{q(x_n,y_n)}{y_n^d}\right\rvert\leq \frac{M}{\lvert y_n \rvert}\leq \frac{\tilde{M}}{{\lvert y\rvert}^{d^n}}
\end{equation}
for $R>0$ sufficiently large. Thus the series (\ref{Log}) converges and consequently the series (\ref{tele}) also converges. 

\medskip
\no 
The function
\[
\phi_H^+(x,y)={c_H}^{-\frac{1}{d-1}}\lim_{n\rightarrow \infty}y_{n}^{\frac{1}{d^{n}}}={c_H}^{-\frac{1}{d-1}}\left(y.\frac{{y_1}^{\frac{1}{d}}}{y}.\cdots.\frac{y_{n+1}^{\frac{1}{d^{n+1}}}}{y_{n}^{\frac{1}{d^{n}}}}\cdots\right)
\]
is clearly well-defined. Further, it follows from (\ref{bott}) and (\ref{zero}) that 
\[
\phi_H^+(x,y)\sim y
\]
as $\lVert (x,y)\rVert\rightarrow \infty$ in $V_R^+$.

\medskip 
\no 
Since
\begin{eqnarray*}
H^{n+1}(x,y)&=&H({(H^n)}_1(x,y), {(H^n)}_2(x,y))\\
&=&({(H^{n+1})}_1(x,y), c_H{(H^n)}_2^d(x,y)+q({(H^n)}_1(x,y),{(H^n)}_2(x,y))),
\end{eqnarray*}
it follows that
\begin{eqnarray*}
{(H^n)}_2(H(x,y))&=& 
c_H {(H^n)}_2^d(x,y)(1+L(x,y))
\end{eqnarray*}
where $L(x,y)\rightarrow 0$ as $\lVert (x,y)\rVert\rightarrow \infty$ in $V_R^+$.

\medskip 
\no
Therefore,
\begin{eqnarray*}
\phi_H^+(H(x,y))&=&{c_H}^{-\frac{1}{d-1}}\lim_{n\rightarrow \infty} (H^n)_2^{\frac{1}{d^n}}(H(x,y))\\
&=&c_H\lim_{n\rightarrow \infty} c_H^{\frac{1}{d^n}} {\left({c_H}^{-\frac{1}{d-1}}{(H^n)}_2^{\frac{1}{d^n}}(x,y){(1+L(x,y))}^{\frac{1}{d^{n+1}}}\right)}^d\\
&=&c_H {\left(\phi_H^+(x,y)\right)}^d.
\end{eqnarray*}

\medskip   
\no 
Consider $H^{-n}(x,y)=({(H^{-n})}_1(x,y),{(H^{-n})}_2(x,y))$ and note that $x_n={(H^{-n})}_1(x,y)$ is a polynomial in $x$ and $y$ of degree $d^{n}$.
As before, consider the telescoping product
\begin{equation*}
x.\frac{{x_1}^{\frac{1}{d}}}{x}.\cdots.\frac{x_{n+1}^{\frac{1}{d^{n+1}}}}{x_{n}^{\frac{1}{d^{n}}}}\cdots
\end{equation*}
which can be shown to be convergent. Thus, we can
define 
\[
\phi_H^-(x,y)={c_H'}^{-\frac{1}{d-1}}\lim_{n\rightarrow \infty}x_{n}^{\frac{1}{d^{n}}}.
\]
The properties of $\phi_H^-$ can be established as in the case of $\phi_H^+$.

\end{proof}
\no 
Since 
\[
G_H^+(x,y)=\lim_{n\rightarrow \infty}\frac{1}{d^n}\log\left \lVert{(H^n)}_2(x,y)\right\rVert
\]
in $V_R^+$, it follows that
\begin{equation}\label{Green1}
G_H^+=\log \lvert\phi_H^+\rvert+\frac{1}{d-1}\log \lvert c_H\rvert
\end{equation}
in $V_R^+$. Similarly
\begin{equation}\label{Green2}
G_H^-=\log \lvert\phi_H^-\rvert+\frac{1}{d-1}\log \left\lvert {c_H'}\right\rvert
\end{equation}
in $V_R^-$.

\begin{rem}
The uniqueness of $\phi_H^\pm$ follows immediately once we establish (\ref{Green1}) and (\ref{Green2}) along with their properties recorded in Proposition \ref{Bot}.
\end{rem}

\section{Proof of Theorem 1.1}

\no The first step is to show that
\begin{center}
{\it $F$ is a polynomial automorphism:}
\end{center}

\medskip

\no Since $F(K_H^\pm)=K_H^\pm$, it follows that $G_H^\pm\circ F=0$ on $K_H^\pm$. Fix $x \in \mathbb{C}$ and consider 
\[
g_{x}(y)=G_H^+\circ F (x,y)
\]
which is harmonic outside the compact set $K_H^+ \cap (\{x\}\times \mathbb{C})$. Since it is harmonic outside a large disk of radius $R>0$, $g_{x}$ has a harmonic conjugate $h_{x}$ in $\{\vert y \vert > R\}$ with period $c_{x}$. Therefore
\[
\psi_{x}(y)=g_{x}(y)-c_{x}\log \lvert y\rvert+ i h_{x}(y)
\]
is holomorphic in $\{\lvert y\rvert >R\}$. Further, note that 
\[
\left\lvert\exp(-\psi_{x}(y))\right\rvert \leq {\lvert y\rvert}^{c_{x}}
\]
which means that $\exp(-\psi_{x}(y))$ has at most a pole at infinity and consequently,
\[
\exp(-\psi_{x}(y))=y^k \exp f(y)
\]
where $f$ is a holomorphic function in $\{\lvert x\rvert >R\}$ having a removable singularity at infinity. Here $k = k_x$ is a positive integer. Taking absolute values and then log, we get 
\[
g_{x}(y)-c_{x}\log \lvert y \rvert=-k \log \lvert y \rvert-{\rm{Re}}(f(y))
\] 
in $\{\lvert y\rvert >R\}$. Therefore, 
\[
g_{x}(y)=b_{x}\log \lvert y\rvert +O(1) 
\]
in $\{\lvert y\rvert >R\}$ and since $g \ge 0$ everywhere,  
\[
g_{x}(y)=b_{x}\log^+ \lvert y\rvert +O(1) 
\]
in $\mathbb{C}$. Since the $O(1)$ term is bounded near infinity, $b_x$ is in fact 
the period of $g_x(y)$. 

\medskip

It turns out that $b_x$ is independent of $x$. To prove this, let us work in a small neighbourhood of a fixed $x_0$ and $R$ is large enough as before. Let $p, q$ be two distinct points near $x_0$ and let $I$ be the straight line segment joining them. Then 
\[
\Sigma = \{(x, y) : x \in I, \vert y \vert = R \}
\]
is a smooth real 2-surface with two boundary components namely, 
\[
\{(p, y) : \vert y \vert = R \} \cup \{ (q, y): \vert y \vert = R \}. 
\]
Thanks to Stokes' theorem,
\[
b_p - b_q = \int_{\partial \Sigma} d^c (G^+_H \circ F) = \int_{\Sigma} dd^c(G^+_H \circ F) = 0
\]
where the last equality holds due to the pluriharmonicity of $G^+_H \circ F$ on $V^+_R$. Hence $b_x$ is locally constant and therefore constant everywhere. Let us write $b_x = b^+$ for all $x \in \mathbb{C}$.

\medskip 

\no 
Now $g_{x_0}(y)=b^+ \log^+ \lvert y \rvert+ O(1)$ and on the other hand $G_H^+(x,y)=\log^+ \lvert y\rvert +O(1)$ in $V_R^+$ (See (3.4) in \cite{BS}). The difference $g_{x_0}(y)-b^+ G_H^+(y)$ is therefore harmonic at each $y$ for which $(x_0,y)\in \mathbb{C}^2\setminus K_H^+$ with a removable singularity at infinity and vanishes for $(x_0,y)\in K_H^+$. Thus $g_{x_0}(y)=b^+G_H^+(x_0,y)$ for each $y\in \mathbb{C}$. 

\medskip 
\no 
The same argument shows that $G_H^+\circ F-b^+G_H^+ \equiv 0$ in $\Delta(x_0;r_0)\times \mathbb{C}$. The difference is pluriharmonic in $\mathbb{C}^2\setminus K_H^+$ which is connected and it vanishes in $K_H^+$. Therefore, we have $G_H^+\circ F=b^+G_H^+$ in $\mathbb{C}^2$. The same proof applies to $G_H^- \circ F$.
Thus
\begin{equation}
G_H^\pm \circ F=b^\pm G_H^\pm 
\end{equation}
everywhere in $\mathbb{C}^2$.

\medskip 
\no 
It remains to show that $F$ has polynomial growth in $\mathbb{C}^2$. Let $(x,y)\in \mathbb{C}^2$ and suppose that $F(x,y)\in \overline{V_R\cup V_R^+}$. Then it follows from (\ref{L1}), (\ref{L2}) and (\ref{L3}) that 
\begin{equation*}
\log \lvert F_2 \rvert-C \leq G_H^+(F(x,y))=b^+ G_H^+(x,y)\leq b^+ (\max \{\log^+\lvert x \rvert, \log^+\lvert y \rvert\}+C).
\end{equation*}
Since $\lvert F_1\rvert\leq \lvert F_2\rvert+R$,  
\begin{equation}\label{Pol1}
\lvert F(x,y)\rvert \leq L \max \{{(\lvert x \rvert+1)}^{b^+}, {(\lvert y \rvert+1)}^{b^+}\}
\end{equation}
for some $L>1$. 

\medskip 
\no 
Similarly, if $F(x,y)\in V_R^-$, then
\begin{equation*}
\log \lvert F_1 \rvert-C \leq G_H^-(F(x,y))=b^- G_H^-(x,y)\leq b^- (\max \{\log^+\lvert x \rvert, \log^+\lvert y \rvert\}+A).
\end{equation*}
Therefore, using $\lvert F_1\rvert\leq \lvert F_2\rvert+R$, we obtain 
\begin{equation}\label{Pol2}
\lvert F(x,y)\rvert \leq L \max \{{(\lvert x \rvert+1)}^{b^+}, {(\lvert y \rvert+1)}^{b^+}\}
\end{equation}
for some $L>1$. 

\medskip
\no 
Combining (\ref{Pol1}) and (\ref{Pol2}), it follows that $F$ has polynomial growth throughout $\mathbb{C}^2$ and hence $F$ is itself a polynomial automorphism.

\medskip

\no {\it Remark:} Starting with the assumption that $F$ preserves $K^{\pm}_H$ (which are exactly the zero sets of $G^{\pm}_H)$, (3.1) shows that 
\[
F : \{G^{\pm}_H < \alpha\} \rightarrow \{G^{\pm}_H < b^{\pm} \alpha \} 
\]
is biholomorphic for every $\alpha > 0$. As mentioned earlier, sub-level sets of $G^{\pm}$ are examples of {\it Short} $\mathbb{C}^2$'s and hence every $F$ that preserves $K^{\pm}_H$ gives rise to a biholomorphism between a pair of {\it Short} $\mathbb{C}^2$'s. This biholomorphism is in fact a polynomial map.

\medskip

There are two possibilities depending on whether $\deg(F) = 1$ or $\deg(F) \ge 2$. We begin with the latter case.

\subsection*{Case 1:} Suppose that $\deg(F) \ge 2$.

\medskip 
\no 
Since $F$ is a polynomial automorphism, Jung's theorem (see \cite{J}) shows that $F$ can be written as a composition of affine and elementary maps in $\mathbb{C}^2$. Recall that an elementary map is of the form
\[
e(x, y) = (\alpha x + p(y), \beta y + \gamma)
\]
where $\alpha \beta \not= 0$ and $p(y)$ is a polynomial in $y$. Thus four cases arise:

\medskip
\no
{\it{Case (i)}}: Let
\[
F=a_1\circ e_1 \circ a_2 \circ e_2\circ \cdots \circ a_k \circ e_k
\]
for some $k\geq 1$ where the $a_i$'s are non-elementary affine maps and the 
$e_i$'s are non-affine elementary maps.
For simplicity, assume that
\[
F=a_1\circ e_1 \circ a_2 \circ e_2.
\]
Let 
\[
a_1(x,y)=(\alpha_1 x+ \beta_1 y+ \delta_1,\alpha_2 x+ \beta_2 y+ \delta_2 ).
\]
for $\alpha_2\neq 0$ and consider the maps
\[
a_1^2(x,y)=(\alpha_2 x+ \beta_2 y+ \delta_2, s_2y+r_2)
\]
and
\begin{equation}\label{aff}
a_1^1(x,y)=(bx+cy,y)
\end{equation}
where $b \not= 0, c=\alpha_1/\alpha_2$, $r_2=(\delta_1-c\delta_2)/b$ and $s_2=(\beta_1-c\beta_2)/b$. For any non-zero $b$, it turns out that $a_1= a_1^1\circ \tau\circ a_1^2$ where $\tau(x,y)=(y,x)$. Writing $a_2$ in a similar fashion, we get
\[
F=a_1^1\tau a_1^2 e_1 a_2^1\tau a_2^2 e_2.
\] 
Note that both $a_1^2 e_1 a_2^1$ and $a_2^2 e_2$ are elementary maps and we denote them by $E_1$ and $E_2$ respectively. Thus 
\[
F=a_1^1 \tau E_1 \tau E_2
\]
where $E_i(x,y)=(m_i x+p_i(y), k_i y+ r_i)$ and the $p_i$'s are polynomials in 
$y$ of degree at least $2$ for $i=1,2$.

\medskip
\no
Assume $c \neq 0$ in (\ref{aff}). Choose an unbounded sequence ${(x_n,y_n)}_{n\geq 1}\subseteq K_H^-\cap V_R^+$. Since both $\tau E_1$ and $\tau E_2$ are H\'{e}non maps, it follows that $(x_n',y_n')=\tau E_1 \tau E_2(x_n, y_n)\in V_R^+$ for all $n\geq 1$. Here and in what follows, $R$ may have to be increased finitely many times so that the filtration works for all H\'{e}non maps that will be encountered. Further, the sequence ${(x_n',y_n')}_{n\geq 1}$ is   unbounded. Take $\lvert b\rvert$ sufficiently small so that $\lvert c\rvert > \lvert b\rvert >0$ and thus
\[
\lvert b x_n'+cy_n'\rvert \geq (\lvert c\rvert-\lvert b\rvert)\lvert y_n'\rvert,
\]
which implies 
\begin{equation}\label{a}
\lvert x_n''\rvert \geq (\lvert c\rvert-\lvert b\rvert)\lvert y_n''\rvert
\end{equation}
where $(x_n'',y_n'')=a_1^1(x_n',y_n')=(b x_n'+ c y_n', y_n')$. Since $(x_n,y_n)\in K_H^-\cap V_R^+$ and 
\[
F(K_H^-)=K_H^-\subseteq V_R \cup V_R^+
\]
it follows that $\Vert (x_n'',y_n'') \Vert \rightarrow \infty$ as $n\rightarrow \infty$ in $V_R^+$. Since $\overline{K_H^-}=K_H^-\cup I^-$ in $\mathbb{P}^2$ (see \cite{DS}), it follows that $(x_n'', y_n'')$ approaches the indeterminacy point $I^-$ as $n\rightarrow \infty$. Consequently, 
\begin{equation}\label{b}
\frac{\lvert x_n''\rvert}{\lvert y_n''\rvert}\rightarrow 0
\end{equation} 
as $n\rightarrow \infty$. Comparing (3.5) and (\ref{b}), $c = 0$. Hence $F$ is a H\'{e}non map.

\medskip
\no 
{\it{Case (ii):}} Let 
\[
F=a_1\circ e_1\circ a_2 \circ\cdots\circ e_{k-1} \circ a_k
\]
for some $k\geq 2$.
For simplicity, assume that $F=a_1\circ e_1 \circ a_2$. As in the previous case we can write
\[
F=a_1^1 \tau a_1^2 e_1 a_2^1 \tau a_2^2
\] 
where $a_1^1(x,y)=(bx+cy,y)$ and $ \tau a_2^2(x,y)=(s_2 y+r_2, \alpha_2 x+ \beta_2 y+ \delta_2)$.

\medskip
\no 
{\it{Claim}:} There exists a sequence ${(x_n,y_n)}_{n \geq 1}\subseteq K_H^+ \cap V_R^-$ with $\lvert x_n \rvert \geq \lvert y_n\rvert\geq n$ for all $n\geq 1$

\medskip
\no
Suppose no such sequence exists. Choose $(x_n,y_n)\in K_H^+\cap V_R^-$ such that  $\lvert y_n \rvert$ is bounded by a fixed constant $M>1$ for all $n\geq 1$  and $\lvert x_n \rvert \rightarrow \infty$ as $n\rightarrow \infty$; of course, $M$ depends on the sequence $(x_n, y_n)$. Without loss of generality, let the H\'{e}non map $H:(x,y)\mapsto (y, p(y)-\delta x)$ with $\delta\neq 0$. Then there exists a subsequence $\{(x_{n_k}, y_{n_k})\}\subset  K_H^+ \cap V_R^-$  such that $\{H(x_{n_k}, y_{n_k})\}\subset  K_H^+ \cap V_R^-$ and consequently,
\[
\lvert y_{n_k} \rvert \geq \lvert p(y_{n_k})-\delta x_{n_k}\rvert  \geq \lvert \delta\rvert\lvert x_{n_k} \rvert-\lvert p(y_{n_k})\rvert.
\]
Since $\lvert y_{n_k}\rvert \leq M$ for all $k\geq 1$, the sequence $\{x_{n_k}\}$ also turns out to be bounded which is a contradiction.

\medskip
\no 
Since $\overline{K_H^+}=K_H^+ \cup I^+$ in $\mathbb{P}^2$, 
\[
\frac{\vert y_n \vert}{\vert x_n \vert}\rightarrow 0 \text{ as } n\rightarrow \infty 
\]
and thus 
\begin{equation}\label{eps}
\lvert y_n \rvert \leq \epsilon_n\lvert x_n\rvert
\end{equation}
for all $n\geq 1$ with $\epsilon_n \rightarrow 0$.

\medskip
\no
Now $\tau a_2^2(x_n,y_n)=(s_2 y_n+r_2, \alpha_2 x_n + \beta_2 y_n + \delta_2)$ and
\begin{eqnarray}
\lvert \alpha_2 x_n + \beta_2 y_n +\delta_2\rvert &\geq& \lvert \alpha_2\rvert \lvert x_n\rvert-\lvert\beta_2\rvert \lvert y_n\rvert-\lvert \delta_2\rvert \nonumber\\
&\geq& (\lvert \alpha_2 \rvert - \epsilon_n \lvert \beta_2\rvert)\lvert x_n\rvert-\lvert\delta_2\rvert \nonumber\\
&\geq& \frac{1}{2}\lvert \alpha_2 \rvert  \lvert x_n \rvert-\lvert \delta_2 \rvert  \nonumber\\
&\geq& \frac{1}{2} \lvert \alpha_2\rvert \lvert y_n \rvert-\lvert \delta_2 \rvert \geq \lvert s_2\rvert \lvert y_n\rvert + \lvert r_2\rvert \lvert y_n\rvert-\lvert \delta_2\rvert  \nonumber
\end{eqnarray}
for all  $n\geq n_0$. The last inequality follows since for $\lvert s_2\rvert$ and $\lvert r_2\rvert$ sufficiently small, one can choose ${\lvert\alpha_2\rvert}/{2}\geq\lvert s_2\rvert+ \lvert r_2\rvert$

\medskip 
\no 
Since $\lvert y_n\rvert \rightarrow \infty$ as $n\rightarrow \infty$, we get
\[
\lvert \alpha_2 x_n + \beta_2 y_n +\delta_2\rvert \geq \lvert s_2 \rvert \lvert y_n\rvert + \lvert r_2\rvert \geq \lvert s_2 y_n + r_2\rvert
\]
and 
\[
\lvert \alpha_2 x_n + \beta_2 y_n +\delta_2\rvert \geq R
\]
for sufficiently large $n$.

\medskip 
\no 
Thus for a sequence ${(x_n,y_n)}_{n\geq 1}\subseteq K_H^+\cap V_R^-$ with $\lvert x_n \rvert \geq \lvert y_n \rvert \geq n$, it turns out that $\tau a_2^2 (x_n,y_n)\in V_R^+$ for sufficiently large $n$. Then 
\[
(x_n',y_n')= \tau a_1^2 e_1 a_2^1 \tau a_2^2 (x_n,y_n)\in V_R^+
\]
and consequently
\[
\lvert b x_n'+c y_n'\rvert\leq (\lvert b\rvert+\lvert c\rvert)\lvert y_n'\rvert.
\]  
Hence
\begin{equation}\label{contra}
\lvert y_n''\rvert \geq \frac{1}{(\lvert b \rvert+ \lvert c \rvert)} \lvert x_n''\rvert
\end{equation}
where $(x_n'',y_n'')=a_1^1(x_n',y_n')$.

\medskip 
\no 
Now since $F(K_H^+)=K_H^+$, 
\[
(x_n'',y_n'')=F(x_n,y_n)\in K_H^+\cap V_R^-
\]
for sufficiently large $n\geq 1$ and $\Vert (x_n'',y_n'')\Vert \rightarrow \infty$ as $n\rightarrow \infty$. By (\ref{eps}), it follows that 
\[
\lvert y_n''\rvert \leq \epsilon_n \lvert x_n''\rvert
\]
where $\epsilon_n\rightarrow 0$ as $n\rightarrow \infty$ which clearly contradicts (\ref{contra}). Thus $F$ cannot be of this form.

\medskip
\no 
{\it{Case (iii):}} Let 
\[
F=e_1\circ a_1 \circ e_2 \circ a_2 \circ \cdots \circ e_k \circ a_k
\]
for some $k\geq 1$. Note that $F^{-1}$ has a form as in Case 1. Since $F^{-1}$ also keeps $K_H^{\pm}$ invariant, it follows that $F^{-1}$ is a H\'{e}non map.

\medskip
\no 
{\it{Case (iv):}} Let 
\[
F=e_1\circ a_1 \circ e_2 \circ a_2 \circ \cdots\circ a_{k-1}\circ e_{k-1} \circ e_k 
\]
for some $k\geq 1$. For simplicity, let us work with 
\[
F=e_1\circ a_1 \circ e_2.
\]
As in the previous cases, we can write
\[
F=e_1 a_1^1\tau a_1^2 e_2
\]
and consequently,
\[
\tau F=\tau e_1 a_1^1\tau a_1^2 e_2.
\]
Note that both $\tau e_1 a_1^1$ and $\tau a_1^2 e_2$ are H\'{e}non maps. 

\medskip 
\no 
Let $(x,y)\in K_H^- \cap V_R^+$, then $\tau F(x,y)\in K_H^- \cap V_R^+$. Thus $F(x,y)\notin K_H^-$ contradicting the hypothesis. Therefore, $F$ cannot be of this form.

\medskip 
\no 
Finally, the conclusion is that if $F(K_H^{\pm})=K_H^\pm$ for some automorphism $F$ in $\mathbb{C}^2$, then either $F$ or $F^{-1}$ must be a  H\'{e}non map. For the sake of definiteness, we assume that $F$ is a H\'{e}non map.

\medskip 

\begin{center}
{\it The Green's functions of $H$ and $F$ coincide:}
\end{center}
\medskip 
\no  
Now $F(K_H^\pm)=K_H^\pm$ and consequently, $F^n(K_H^+)=K_H^+ \subseteq V_R \cup V_R^-$ for all $n\geq 1$.  Since the forward orbit of any $z\in \mathbb{C}^2\setminus K_F^-$ (See (\ref{escape})) eventually enters in $V_R^+$ and converges to the point at infinity $I^-$, we see that
\[
K_H^+\subseteq K_F^+.
\]
A rigidity theorem of Dinh-Sibony (\cite{DS}) asserts that the $dd^c$-closed currents $\mu_H^+$ and $\mu_F^+$ which are both supported on $K^+_F$ must agree up to a multiplicative constant. Therefore,  
\[ 
J_H^+ = \text{supp} \; \mu_H^+=\text{supp} \; \mu_F^+ = J_F^+.
\]
Hence, the pluricomplex Green functions of these sets (namely, $J_H^+$ and $J_F^+$) are the same, i.e., 
\[
G_H^+=G_F^+
\]
and in the same vein, it can be shown that
\[
G_H^-=G_F^-
\]
in $\mathbb{C}^2$.

\medskip 
\no 
Now by (\ref{Green1}) 
\begin{equation*}
G_H^+=\log \lvert \phi_H^+\rvert+\frac{1}{d_H-1}\log \lvert c_H \rvert \text{ and } G_F^+=\log \lvert \phi_F^+\rvert +\frac{1}{d_F-1}\log \lvert c_F\rvert.
\end{equation*}
on $V_R^+$.

\medskip 
\no  
Since $\phi_H^+$ and $\phi_F^+$ are both asymptotic to $y$ as $\lVert(x,y)\rVert\rightarrow \infty$ in $V_R^+$, it follows that 
\begin{equation}\label{const}
\frac{1}{d_H-1}\log \lvert c_H \rvert=\frac{1}{d_F-1}\log \lvert c_F \rvert
\end{equation}
and consequently 
\begin{equation*}\label{Phi+}
\phi_H^+ \equiv \phi_F^+
\end{equation*}
in $V_R^+$. Similarly,
\begin{equation*}\label{Phi-}
\phi_H^- \equiv \phi_F^-
\end{equation*}
in $V_R^-$. From now on we shall write $\phi^\pm$ for $\phi_H^\pm \equiv \phi_F^\pm$.

\medskip 
\no 
By (\ref{const}),
\[
c_H^{d_F}c_F=c_F^{d_H}c_H\delta_+ 
\]
for some $\delta_+$ with $\lvert \delta_+\rvert=1$.

\medskip 
\no 
By Proposition \ref{Bot},  
\begin{equation*}
\phi^+ \circ F \circ H(x,y)=c_F{(\phi^+ \circ H (x,y))}^{d_F}=c_F c_H^{d_F}{(\phi^+(x,y))}^{d_H d_F}
\end{equation*}
and similarly,
\begin{equation*}
\phi^+ \circ H \circ F(x,y)=c_H{(\phi^+ \circ F (x,y))}^{d_H}=c_H c_F^{d_H}{(\phi^+(x,y))}^{d_H d_F}.
\end{equation*}
Therefore, 
\begin{equation*}\label{equality}
\phi^+ ( F \circ H) =\delta_+ \phi^+ ( H \circ F)
\end{equation*}
on $V_R^+$. 
Since
\begin{equation*}
\phi^+ \circ F \circ H(x,y)\sim {(F\circ H)}_2(x,y)  
\end{equation*}
and 
\begin{equation*}
\phi^+ \circ H \circ F(x,y)\sim {(H\circ F)}_2(x,y)  
\end{equation*}
as $\lVert(x,y)\rVert\rightarrow \infty$ in $V_R^+$, it follows that for a fix $x_0 \in \mathbb{C}$, 
\[
{(F\circ H)}_2(x_0,y)- \delta_+{(H\circ F)}_2(x_0,y)\sim 0 \text{ as } \lvert y \rvert \rightarrow \infty.
\]
The expression on the left is a polynomial in $y$ and hence
\[
{(F\circ H)}_2(x_0,y)=\delta_+{(H\circ F)}_2(x_0,y)
\]
for all $y\in \mathbb{C}$. Therefore,
\begin{equation}\label{rel1}
{(F\circ H)}_2 \equiv \delta_+{(H\circ F)}_2
\end{equation}
in $\mathbb{C}^2$. To study the first components of these maps, we again appeal to Proposition \ref{Bot} which shows that 
\begin{equation*}
(\phi^-\circ H(x,y))={\left(\frac{1}{c_H'}\right)}^{\frac{1}{d_H}}{\left(\phi^-(x,y)\right)}^{\frac{1}{d_H}}
\end{equation*}
for $(x,y)\in H^{-1}(V_R^-)$ 
and 
\begin{equation*}
(\phi^-\circ F(x,y))={\left(\frac{1}{c_F'}\right)}^{\frac{1}{d_F}}{\left(\phi^-(x,y)\right)}^{\frac{1}{d_F}}
\end{equation*}
for $(x,y)\in F^{-1}(V_R^-)$, where principal roots are being taken. Thus 
\begin{equation*}
\phi^-\circ H \circ F(x,y)={\left( \frac{1}{c_H'}\right)}^{\frac{1}{d_H}}{\left( \frac{1}{c_F'}\right)}^{\frac{1}{d_H d_F}}{\left( \phi^-(x,y)\right)}^{\frac{1}{d_H d_F}}
\end{equation*}
and 
\begin{equation*}
\phi^-\circ F \circ H(x,y)={\left( \frac{1}{c_F'}\right)}^{\frac{1}{d_F}}{\left( \frac{1}{c_H'}\right)}^{\frac{1}{d_H d_F}}{\left( \phi^-(x,y)\right)}^{\frac{1}{d_H d_F}}
\end{equation*}
for all $(x,y)\in {(H\circ F)}^{-1}(V_R^-)\cap {(F\circ H)}^{-1}(V_R^-)=U $, say. Note that $U$ is an non-empty open neighborhood of $I^+=[1:0:0]$ in $\mathbb{P}^2$.

\medskip 
\no 
Thus
\[
\phi^-\circ (H\circ F)=\delta_- (F\circ H)
\]
for all $z\in U$ where
\begin{equation*}
{\left(\frac{1}{c_H'}\right)}^{\frac{1}{d_H}}{\left(\frac{1}{c_F'}\right)}^{\frac{1}{d_H d_F}}=\delta_- {\left(\frac{1}{c_F'}\right)}^{\frac{1}{d_F}}{\left(\frac{1}{c_H'}\right)}^{\frac{1}{d_H d_F}}. 
\end{equation*}
Therefore,
\begin{equation*}\label{delta1}
{(c_H')}^{d_F}(c_F'){(\delta_-)}^{d_H d_F} = {(c_F')}^{d_H}(c_H').
\end{equation*} 
By (\ref{Green2}), it follows that
\begin{equation*}
\lvert \delta_-\rvert =1
\end{equation*}
as in the previous case.

\medskip 
\no 
Now for a fixed $c\neq 0$, there exists $\epsilon>0$ sufficiently small such that
\[
U_{\epsilon,c}=\{[{1}/{y}:c:1]:0\neq \lvert y\rvert <\epsilon\}
\]
is contained inside $U\subset V_R^-$ with $I^+$ as a boundary point. Therefore, we can choose a sequence $[x_n:c:1]\in U_{\epsilon,c}$ with $x_n\rightarrow \infty$. Because $(F\circ H)(x_n,c), (H\circ F)(x_n,c)\in V_R^-$ as a consequence of  $(x_n,c)$ being in $U$, we have  
$$
{(F\circ H)}_1(x_n,c),{(H\circ F)}_1(x_n,c)\rightarrow \infty
$$
as $n\rightarrow \infty$. 

\medskip
\no 
Now since $\phi^-(x,y)\sim x$ as $\lVert(x,y)\rVert\rightarrow \infty$,  
\begin{equation*}
{(H\circ F)}_1(x_n,c)-\delta_-{(F\circ H)}_1(x_n,c)\rightarrow 0
\end{equation*} 
as $n\rightarrow \infty$. The expression on the left is a polynomial in $x$ for each fixed $c$ and thus
\begin{equation*}
{(H\circ F)}_1(x,c)=\delta_-{(F\circ H)}_1(x,c)
\end{equation*}
for all $x\in \mathbb{C}$. Thus using the same argument as in the previous case, we get
\begin{equation}\label{rel2}
{(F\circ H)}_1 \equiv \delta_-{(H\circ F)}_1
\end{equation}
in $\mathbb{C}^2$.

\medskip 
\no 
Hence using (\ref{rel1}) and (\ref{rel2}), we get
\[
F\circ H= C\circ H\circ F
\]
where $C(x, y) = (\delta_- x,\delta_+ y)$ with $\lvert \delta_\pm\rvert=1$.

\medskip 
\no 
Similarly, if $F^{-1}$ is a H\'{e}non map, then
\[
F^{-1}\circ H= C\circ H\circ F^{-1}
\]
with $C(x, y) = (\delta_- x,\delta_+ y)$, $\lvert \delta_{\pm} \vert = 1$ as above.

\medskip 
\begin{center} 
{\it Some iterates of $F$ and $H$ agree upto an affine automorphism:}
\end{center}
 \medskip 
 \no 
Let $d_F, d_H$ be the degrees of $F, H$ respectively. We use the following fact (See Theorem 1.5, \cite{Bi}): {\it If for $m, n \in \mathbb{N}$, the condition that $d_F^m\leq d_H^n$ implies that $d_F^m$ divides $d_H^n$, then there
exist $n_0,m_0 \in \mathbb{N}$ such that $d_F^{m_0}=d_H^{n_0}$.}

\medskip
\noindent

Suppose that $d_F^m \le d_H^n$ for some $m, n \in \mathbb{N}$. Then
\[
L=H^n \circ F^{-m} 
\]
satisfies $L(K_H^\pm)=K_H^\pm$. By the aforementioned arguments, either $L$ or $L^{-1}$ is a H\'{e}non map. We claim that $L^{-1}$ cannot be a H\'{e}non map for if it were, then
\begin{equation}\label{LH}
G^{\pm}_{L^{-1}}=G_H^\pm
\end{equation}
and
\begin{equation*}
G^+_{L^{-1}} \circ L^{-1}= d_{L^{-1}} G^+_{L^{-1}}
\end{equation*}
where $d_{L^{-1}}$ is the degree of $L^{-1}$. Using (\ref{LH}) and the fact that $G^{\pm}_H = G^{\pm}_F$, 
\begin{equation*}
G^+_{L^{-1}} \circ L^{-1} = G_H^+ \circ F^m \circ H^{-n}=\frac{d_F^m}{d_H^n}G_H^+
\end{equation*}
in $\mathbb{C}^2$. Thus 
\[
d_{L^{-1}} = \frac{d_F^m}{d_H^n} \le
 1
\]
which is a contradiction. Hence $L$ is a H\'{e}non map. Let $d_L$ be the degree of $L$. The same reasoning applied to $G^{\pm}_L$ now shows that
\[
d_L = \frac{d_H^n}{d_F^m}
\]
and hence $d_F^m$ divides $d_H^n$. Thus there exist $m_0, n_0 \in \mathbb{N}$ such that $d_F^{m_0}=d_H^{n_0}$.  

\medskip

Let $L_{m_0, n_0} = F^{m_0} \circ H^{-n_0}$. Now two cases arise: the degree of $L_{m_0, n_0}$ is either $1$ or at least $2$. In the former case, there is nothing to prove. In the latter case, either $L_{m_0, n_0}$ or its inverse is a H\'{e}non map. If $L_{m_0, n_0}$ is H\'{e}non, then  
\[
\deg(F^{m_0}) = \deg(L_{m_0, n_0}) \cdot \deg(H^{n_0})
\]
shows that the degree of $L_{m_0, n_0}$ is $1$. If the inverse of $L_{m_0, n_0}$ is H\'{e}non, then running the same argument with
\[
H^{n_0} = L^{-1}_{m_0, n_0} \circ F^{m_0}
\]
shows that the degree of $L_{m_0, n_0}$ is $1$ again. In all cases then, $F^{m_0} = \sigma \circ H^{n_0}$ for some affine automorphism $\sigma$.
\medskip

\no \subsection*{Case 2:} Suppose that $\deg(F) = 1$.

\medskip 
\no 
Let $F(x,y)=(cx+ay+d, bx+ey+f)$. Now $F(K_H^\pm)=K_H^\pm$, thus $(F\circ H)(K_H^\pm)=K_H^\pm$.  Since $\deg(F\circ H)\geq 2$, either $(F\circ H)$ or $H^{-1}\circ F^{-1}$ is H\'enon. A degree count shows that  $H^{-1}\circ F^{-1}$ is never a H\'{e}non map. Thus, $(F\circ H)$ is a H\'{e}non map and consequently $a=0$. Now $F$ extends holomorphically to $\mathbb{P}^2$ and $F([1:0:0])=[1:0:0]$. Thus $b=0$ and consequently, $F(x,y)=(cx+d, ey+f)$.
 
\medskip 
 
\begin{center}
{\it{There exists a H\'{e}non map $R$ which generates every $F$ that preserves $K^{\pm}_H$}:}
\end{center}
\medskip 
\no 
Suppose $F(K_H^\pm)=K_H^\pm$ and $F$ is not affine. Without loss of generality, assume $F$ to be a H\'{e}non map. Let $R$ be a H\'{e}non map such that $R(K_H^\pm)=K_H^\pm$ and $\deg(R)\leq \deg(F)$ for all $F$ (non-affine) which preserve $K_H^\pm$. Now since $$
J_H^\pm=J_F^\pm=J_R^\pm,
$$ 
applying a same arguments used previously, it follows that  there exist $m_0$ and $n_0$ such that 
\[
F^{m_0}=\gamma_F R^{n_0}
\]
for some affine automorphism $\gamma_F$. Hence  ${\deg(F)}^{m_0}={\deg(R)}^{n_0}$  and since $\deg(F)\geq \deg(R)$, it follows that $n_0\geq m_0$.  Thus $\deg(R)$ devides $\deg(F)$. Let $\deg(F)=r_F\deg(R) $. consider the automorphism 
\[
\sigma_F=F\circ R^{-r_F }.
\]
Clearly, $\sigma_F (K_H^\pm)=K_H^\pm$. Therefore $\sigma_F$ turns out to be affine using the same set of arguments as in the previous case. Thus 
\[
F=\sigma_F \circ R^{r_F}.
\]

\medskip 
\no 
For the converse, first note that $C\circ H$ is also a H\'{e}non map. Since 
\[
C\circ H(K_H^+)=K_H^+,
\]
it follows that $K_H^+ \subset K_{(C\circ H)}^+$.  Also,
\[
H(K_{(C\circ H)}^+)=C^{-1}(K_{(C\circ H)}^+)
\]
which in turn gives $K_{(C\circ H)}^+ \subset K_H^+$. Thus 
\begin{equation*}
K_{(C\circ H)}^+=K_H^+. 
\end{equation*}
\medskip 
\no 
Since $F\circ H=C\circ H \circ F$, it follows that 
\[
(C\circ H)(F(K_H^+))=F(K_H^+).
\]
Thus 
\begin{equation}\label{1}
F(K_H^+)\subset K_{(C\circ H)}^+=K_H^+ .
\end{equation}
Indeed, if not, then there exists $F(p)\in F(K_H^+)$ such that $F(p)\notin K_{(C\circ H)}^+$. Therefore, 
\[
\Vert {(C\circ H)}^n (F(p)) \Vert \rightarrow \infty
\]
in $V_R^+$. Also, ${(C\circ H)}^n (F(p))=F(p_n)$ with $p_n\in K_H^+$ for all $n\geq 1$ and $\Vert p_n \Vert \rightarrow \infty$ in $V_R^-$. Thus, on the one hand, $F(p_n) \in K_H^+ \cap V_R^+$ for all $n \ge 1$ and on the other, $\Vert F(p_n) \Vert \rightarrow \infty$. This is not possible since $K^+_H \subset V_R \cup V_R^-$.

\medskip 
\no 
Further, since $F\circ H=C\circ H \circ F$, it follows that 
\[
H(F^{-1}(K_H^+))=F^{-1}(K_H^+)
\]
and therefore,
$F^{-1}(K_H^+)\subset K_H^+$, i.e., 
\begin{equation}\label{2}
K_H^+\subset F(K_H^+).
\end{equation}
Combining (\ref{1})and (\ref{2}), we get that $F(K_H^+)=K_H^+$. Similarly, it follows that $F(K_H^-)=K_H^-$. Thus 
\[
F(K_H^\pm)=K_H^\pm
\]
and this completes the proof of Theorem 1.1

\section{Two remarks on \short{2}'s}

\no In this section, we will (i) show the existence of a continuum of biholomorphically non--equivalent \short{2}'s and (ii) demonstrate the existence of \short{2}'s that are not Reinhardt and not even biholomorphic to Reinhardt domains.

\medskip

\no We begin with the first theme. Let $H_1$ and $H_2$ be hyperbolic H\'{e}non maps with only one attracting periodic point at the origin. Then ${\rm int}(K_{H_1}^+)$ and ${\rm int}(K_{H_2}^+)$ are Fatou--Bieberbach domains containing the origin. For $c > 0$, the sub-level sets 
\[ 
\Om_1^c=\{z \in \mbb C^2: G_{H_1}^+(z)< c\}
\]
and
\[ 
\Om_2^c=\{z \in \mbb C^2: G_{H_2}^+(z)<c\}
\]
are \short{2}'s and it is evident that ${\rm int}(K_{H_i}^+) \subset \Om_i^c$ for $i = 1, 2$.

\begin{prop}\label{biholomorphism}
Fix $c > 0$ and suppose that $\phi : \Om_1^c \rightarrow \Om_2^c$ is a biholomorphism. Then $\phi(K_{H_1}^+)=K_{H_2}^+.$
\end{prop}

\begin{proof}
Let $\Om_1= {\rm int}(K_{H_1}^+)$ and $\Om_2= {\rm int}(K_{H_2}^+)$ be the Fatou--Bieberbach domains containing the origin corresponding to the maps $H_1$ and $H_2$. Also,
\begin{align}\label{Julia}
J_1^+=\partial \Om_1=\partial K_{H_1}^+ \text{ and }J_2^+= \partial \Om_2=\partial K_{H_2}^+.
\end{align}

\no Note that $\phi(\Om_1)$ is Fatou--Bieberbach domain in $\Om_2^c$. Since $G_{H_2}^+$ is bounded above (by $c$) on it, it follows that $G_{H_2}^+$ must be constant there. Suppose that $G_{H_2}^+ \equiv \alpha$ on $\phi(\Om_1)$. Now suppose there exists $z \in \phi(\Om_1) \setminus K_{H_2^+}$ such that $G_{H_2}^+(z)=\alpha >0$. Since $K_{H_2}^+$ is closed, there exists a ball contained in 
$\phi(\Om_1) \setminus K_{H_2^+}$ on which $G_{H_2}^+ \equiv \alpha$. This is not possible since $G_{H_2}^+$ is a non--constant pluriharmonic function on $\mbb C^2 \setminus K_{H_2}^+$. Hence $\phi(K_{H_1}^+)  \subset K_{H_2}^+.$ A similar argument applied to $\phi^{-1}$ shows that $\phi^{-1}(K_{H_2}^+) \subset K_{H_1}^+$ and hence $\phi(K_{H_1}^+)=K_{H_2}^+$. From \ref{Julia}, we also see that 
$\phi(\Om_1)=\Om_2.$
\end{proof}

\begin{thm}\label{continuum}
There exists a continuum of biholomorphically non--equivalent \short{2}'s. 
\end{thm}
\begin{proof}
Corollary 6.8 of \cite{Wolf} (see \cite{BS3} also) shows that for $s$ in the open interval $(3,4)$, there exist $a = a(s), c = c(s) \in C_H$ ($C_H$ is the main cardioid of Mandelbrot set) such that the H\'{e}non map
\[
H_{a,c}(z,w)=(aw+z^2+c,z)
\]
is hyperbolic, has only one attracting fixed point, say $p_{(a,c)}$ in $\mbb C^2$ and if $\Om_{a, c}$ is the associated basin of attraction, then $J^+_{a,c}=\partial\Om_{a,c}$ and the Hausdorff dimension $J_{a,c}^+$ is $s$.

\medskip\no 
Let $G^+_{a,c}$ denote the Green's function corresponding to the H\'{e}non maps $H_{a,c}.$ Fix $C >0$ and consider the class of \short{2}'s given by
\[ \Om_{a,c}^{C}=\{ z \in \mbb C^2: G_{a,c}^+< C\}.\]
where $a=a(s)$ and $c=c(s)$ for $s \in (3,4).$

\medskip\no 
Let $s_0 \not= s_1 \in (3,4)$ and set $a_0=a(s_0),c_0=c(s_0)$ and $a_1=a(s_1),c_1=c(s_1)$ as above.

\medskip\no 
\textit{Claim:} $\Om_{a_0,c_0}^{C}$ is not biholomorphically equivalent to $\Om_{a_1,c_1}^{C}.$ 

\medskip\no 
Any biholomorphism $\phi: \Om_{a_0,c_0}^{C} \to \Om_{a_1,c_1}^{C}$ must satisfy $\phi(\Om_{a_0,c_0})=\Om_{a_1,c_1}$ by Proposition \ref{biholomorphism} and therefore
\[
\phi(J_{a_0,c_0}^+)=\phi(\partial\Om_{a_0,c_0})=\partial\Om_{a_1,c_1}=J_{a_1,c_1}^+
\]
since $\phi$ is biholomorphic near $J_{a_0,c_0}^+$. Then it must  be true that
\[ 
s_1=\text{dim}_H J_{a_1,c_1}^+=\text{dim}_H \phi(J_{a_0,c_0}^+)=\text{dim}_H J_{a_0,c_0}^+=s_0,
\]
which is a contradiction.

\end{proof}
\no 
Let $H$ be a H\'{e}non map as above and $G_H^\pm$ be the Green's functions of 
$H$. The sub-level sets of $G_H^+$
\[
\Omega_c=\{z\in \mathbb{C}^2: G_H^+ < c\}.
\]
are known to be \short{2}'s by \cite{Fo}.

\begin{prop}
$\Omega_c$ is not Reinhardt for any $c>0$.	
\end{prop}

\begin{proof}
Let us assume that $\Omega_c$ is a Reinhardt domain for some $c>0$. Then by Proposition \ref{biholomorphism}, any torus action preserving $\Omega_c$ also preserves $K_H^+$.  Let $(a,b)\in K_H^+$. Then since $K_H^+$ is also Reinhardt, the circle of radius $\lvert a \rvert$  lying in the horizontal line $y=b$ is in $K_H^+$.  Consider the disc $\mathcal{D}_b=\{(x,y):\lvert x\rvert \leq \lvert a \rvert, y=b\}$. Since $G_H^+$ restricted on the horizontal line $y=b$ is a subharmonic function  and $G_H^+$ vanishes on the boundary of $\mathcal{D}_b$, the maximum principle shows that $G_H^+$ vanishes on all of $\mathcal{D}_b$. Thus $\mathcal{D}_b \subset K_H^+$. Now if we can choose $b$ large enough, in particular $\lvert b\rvert >R$, then $\mathcal{D}_b$ and thus $K_H^+$ intersects $V_R^+$ which is clearly a contradiction. That points $(a, b) \in K_H^+$ exist with $\vert b \vert$ large enough is explained in Case $2$ of the subsection in which $F$ is shown to be a H\'{e}non map.
\end{proof}

\no The question that remains is whether $\Om_c$ can be biholomorphic to a Reinhardt domain. We do not know if this is possible, but it turns out that the \short{2}'s arising as sub-level sets of $H_{a, c}$ (as defined above) cannot be biholomorphic to a Reinhardt domain.

\medskip

\no To show this, recall that pseudoconvex Reinhardt domains are logarithmically convex. In fact, more is true -- they are also locally convexifiable near almost every point on their boundary. Here, a boundary point $p$ of a pseudoconvex Reinhardt domain, say $D \subset \mathbb{C}^k$, is said to be locally convexifiable if there is a neighbourhood $U$ of $p$ in $\mbb C^k$ and a biholomorphism $\phi$ from $U$ onto its image such that $\phi(U \cap D) \subset \mbb C^k$ is convex. Note that no assumptions are being made about the smoothness of $\pa D$ near $p$. We include a proof of this for the sake of completeness.

\begin{prop}\label{Green_Reinhardt}
Let $D \subset \mbb C^k$, $k \ge 2$, be a pseudoconvex Reinhardt domain. Then $D$ is locally convexifiable near every boundary point except possibly for a set of measure zero.
\end{prop}
\begin{proof}
For $1 \le j \le k$, let $Z_j=\{(z_1,\ldots,z_k) \in \mbb C^k: z_j=0\}$. Let 
\[ 
Z=\bigcup_{j=1}^k Z_j
\]
be the union of these hyperplanes.

\medskip

\no \textit{Claim:} $D$ is locally convexifiable near every $z_0=(z_1^0,\hdots,z_k^0) \in \partial D \setminus Z$.

\medskip\no 
For every $1 \le j \le k$, choose $w_j \in \mbb C$ such that $e^{w_j}=z_j^0$. There exists a sufficiently small $r_j>0$ and a branch of the multi-valued logarithm, say ${\rm Log}_j$ defined on the disc 
$\Delta(w_j, r_j) \subset \mbb C$ such that
\begin{align}\label{log}
{\rm Log}_j\Big({\rm exp}\big(\Delta(w_j, r_j)\big)\Big)=\Delta(w_j,r_j)\text{ and } {\rm Log}_j\big({\rm exp}(z)\big)=z
\end{align}
for every $z \in \Delta(w_j, r_j)$. Let $R=(r_1,\hdots,r_k)$ and $\Delta^k(w;R)$ the polydisc at the point $w=(w_1,\hdots,w_k)$. Let 
\[ 
\psi(z_1,\hdots,z_k)=(\exp z_1,\hdots,\exp z_k)
\]
and note that $V=\psi(\Delta^k(w;R))$ is a neighbourhood of $z_0$ due to the local injectivity of $\psi$. Let $W=V \cap D$ and define $\phi$ on $V$ as
\[ 
\phi(z_1,\hdots,z_k)=\big({\rm Log}_1(z_1),\hdots,{\rm Log}_k(z_k)\big).
\]
From (\ref{log}), $\phi$ is injective on $V$. To prove the claim, it is enough to show that $\phi(W)$ is a convex open set in $\mbb C^k.$

\medskip\no 
Let $p,q \in \phi(W)$ where $p=(p_1,\hdots,p_k)$ and $q=(q_1,\hdots,q_k)$. Then there exist $p',q' \in W$ such that $p=\phi(p'),q=\phi(q')$. Let 
\[
s_1=(\log|p'_1|,\hdots,\log|p'_k|) \text{ and }s_2=(\log|q'_1|,\hdots,\log|q'_k|).
\]
But $|p_j'|=\exp(\Re p_j)$ and $|q_j'|=\exp(\Re q_j)$ for every $1 \le j \le k$ and hence
\[
s_1=(\Re p_1,\hdots, \Re p_k) \text{ and }s_2=(\Re q_1,\hdots, \Re q_k).
\]
For $\la \in [0,1]$ let 
\[
s_\la=\la s_1+(1-\la)s_2 \text { and }p_\la=\la p+(1-\la)q.
\]
Then $\Re p_\la=s_\la$ and $s_\la \in \log D$ where  
\[ 
\log D =\{(\log |z_1|,\hdots,\log|z_k|):(z_1,\hdots,z_k) \in D\}.
\]
This means there exists $z_\la=(z_1^\la,\hdots,z_k^\la) \in \Om$ such that 
$s^\la_j=\log|z_j^\la|$ for $1 \le j \le k$ where 
$s_\la=(s_1^\la,\hdots,s_k^\la)$. Since $p,q \in \phi(W) \subset \Delta^k(w;R)$ and $\Delta^k(w;R)$ is convex, $p_\la \in \Delta^k(w;R)$ for all $\lambda \in [0,1]$. Let $\psi(p_\la) =(\tilde{p}_1^\la,\hdots\tilde{p}_k^\la)\in V$. Now 
\begin{align}\label{la point}
\log|\tilde{p}_j^\la|=\log \exp (\Re p_j^\la)=\Re p_j^\la=s_j^\la.
\end{align} 
As $D$ is Reinhardt, $s_\la \in \log D$ implies that 
\[ 
\{(z_1,\hdots,z_k) \in \mbb C^k: \log|z_j|=s_j^\la \text{ for every } 1 \le j \le k\} \subset D.
\]
By (\ref{la point}), $\psi(p_\la) \in D$ and hence $p_\la \in \phi(W)$ for all $\la \in [0,1]$. This proves the claim.
\end{proof}

\begin{thm}
There exist \short{2}'s that are not biholomorphically equivalent to a Reinhardt domain.
\end{thm}
\begin{proof}
Let $\Om_{a(s),c(s)}^{C_0}$ be the \short{2}'s considered in the proof of Theorem \ref{continuum}. Recall that ${\rm dim}_H(J^+_{a(s),c(s)})=s$, $s \in (3,4).$ 

\medskip\no 
For a fixed $s \in (3,4)$ and $C>0$, suppose that there exists a biholomorphism
\[
\phi : \Om_{a(s),c(s)}^{C} \ra G
\]  
where $G \subset \mbb C^2$ is a Reinhardt domain. Let 
\[
D=\phi\big({\rm int}(K^+_{a(s),c(s)})\big) \subset G. 
\]
Now $\phi$ induces a $\mbb T^2$-action on $\Om_{a(s),c(s)}^{C}$ which must preserve ${\rm int}(K^+_{a(s),c(s)})$ by Proposition \ref{biholomorphism}. It follows that $D$ must itself be a Reinhardt domain. $D$ is also pseudoconvex being a Fatou-Bieberbach domain. From Proposition \ref{Green_Reinhardt}, $D$ is  locally convexifiable near almost every boundary point. This implies that the Hausdorff dimension of $\pa D$ is $3$ almost everywhere. But since
\[
J^+_{a(s),c(s)}=\partial K^+_{a(s),c(s)}=\phi^{-1}(\partial D)
\]
and $\phi^{-1}$ is biholomorphic near $\pa D$, the same is true for 
$J^+_{a(s),c(s)}$. This cannot be true since its Hausdorff dimension is $s > 3$ at every point.
\end{proof}

\section{Appendix}
\no 
Let $P$ be a polynomial of degree $d\geq 2$ in the complex plane and let $K_P$ be the filled Julia set of $P$, i.e., 
\[
K_P=\{z\in \mathbb{C}:  \text{the sequence} \; \{P^n(z)\}\text{ is bounded}\}.
\]
\begin{thm}
Let $Q$ be an entire function satisfying $Q(K_P) = K_P$. Then $Q$ is a polynomial whose Julia set coincides with that of $P$.
\end{thm}

To prove this, observe that a similar set of arguments as in the first part of the proof of Theorem {\ref{rigidity}} shows that if $g$ is the Green's function of $\mbb C \setminus K_P$ with pole at infinity, then $g \circ Q$ is a constant multiple of $g$. Combining this with the fact that $g$ behaves like 
$\log \vert z \vert$ near $z = \infty$, it follows that $Q$ has polynomial growth and hence must be a polynomial.

\medskip

Thus, it suffices to consider the following situation: $P$ and $Q$ are polynomials such that 
\[
Q(K_P)=K_P.
\]

\medskip 
\no 
Let $z\in \partial K_P=J_P$, then  there exist sequences $\{z_n\}\in {\rm{int}}(K_P)$ and $\{w_n\}\in I_P$ such that $z_n, w_n\rightarrow z$ and thus $Q(z_n), Q(w_n)\rightarrow Q(z)$ as $n\rightarrow \infty$. Now note that $Q(I_P)=I_P$ where $I_P$ is the unbounded component of the Fatou set $F_P$ of $P$. Further, by open mapping theorem $Q({\rm{int}}(K_P))\subset {\rm{int}}(K_P)$. Thus $Q(z)\in J_P$ and we get that
\[
Q(J_P)\subset J_P.
\]
Let $z\in J_P$, then there exists $y\in K_P$ such that $Q(y)=z$. Again by open mapping theorem since $Q({\rm{int}}(K_P))\subset {\rm{int}}(K_P)$, we have 
\[
J_P\subset Q(J_P).
\]
Therefore,
\[
Q(J_P)=J_P
\]
and since $J_Q$ is the smallest closed invariant (under $Q$) subset in $\mathbb{C}$, we get
\begin{equation}\label{J}
J_Q\subset J_P.
\end{equation}
Now we  have started with the assumption that $Q(K_P)=K_P$. Thus $K_P\subset K_Q$. Let there exists some $z \in K_Q$ such that $z \notin K_P$, then by (\ref{J}), there exists a small ball $B_z$ with center at $z$ such that $B_z\subset  \rm{int}(K_Q)$ and $B_z\cap K_P=\emptyset$. This contradicts (\ref{J}). Thus 
$K_P=K_Q$ and consequently, $J_P=J_Q$. Conversely, if $J_P=J_Q$, then $K_P= K_Q$.

\medskip 
\no 
To conclude, the assumption $Q(K_P)=K_P$ implies that $J_P=J_Q$. In fact, these are equivalent if both $P$ and $Q$ are polynomials.  As indicated in the introduction, $J_P$ and $J_Q$ coincide precisely when
\[
P\circ Q=\sigma \circ Q \circ P
\] 
where $\sigma:z\mapsto az+b$ with $\lvert a\rvert=1$ and $\sigma (J_P)=J_P$.


\end{document}